\documentclass[10pt]{amsart}
\usepackage{amsfonts}
\usepackage{amsbsy}
\usepackage{tikz}
\usepackage{amsmath}
\usepackage{amssymb}
\usepackage{amsthm}
\usepackage{syntonly}
\usepackage{url} 
\usepackage{amscd} 
\usepackage{tikz-cd} 
\usepackage{bm} 
\usepackage{mathdots} 

\usepackage[colorlinks,linkcolor=orange, anchorcolor=green, citecolor=purple ]{hyperref}
\usepackage{graphicx} 

\usepackage{mathrsfs} 
\usepackage{comment} 
\usepackage{bbm} 

\usepackage[utf8x]{inputenc}
    \usepackage{ucs}
\newcommand\quotient[2]{
        \mathchoice
            {
                \text{\raise1ex\hbox{$\#1$}\Big/\lower1ex\hbox{$\#2$}}%
            }
            {
                \#1\,/\,\#2
            }
            {
                \#1\,/\,\#2
            }
            {
                \#1\,/\,\#2
            }
    }

\newcommand{\R}{\mathbb{R}} 
\newcommand{\Z}{\mathbb{Z}}
\newcommand{\Q}{\mathbb{Q}}
\newcommand{\C}{\mathbb{C}}

\newcommand{\bbP}{\mathbb{P}}

\newcommand{\bmA}{\bm{\mathrm{A}}}

\newcommand{\bmB}{\bm{\mathrm{B}}}
\newcommand{\bmc}{\bm{\mathrm{c}}}

\newcommand{\bmf}{\bm{f}}

\newcommand{\bmG}{\bm{\mathrm{G}}}
\newcommand{\bmH}{\bm{\mathrm{H}}}
\newcommand{\bmL}{\bm{\mathrm{L}}}
\newcommand{\bmM}{\bm{\mathrm{M}}}
\newcommand{\bmN}{\bm{\mathrm{N}}}

\newcommand{\bmS}{\bm{\mathrm{S}}}

\newcommand{\bmU}{\bm{\mathrm{U}}}
\newcommand{\bmv}{\bm{\mathrm{v}}}
\newcommand{\bmV}{\bm{\mathrm{V}}}
\newcommand{\bmw}{\bm{\mathrm{w}}}

\newcommand{\bmX}{\bm{\mathrm{X}}}

\newcommand{\bmZ}{\bm{\mathrm{Z}}}

\newcommand{\rmD}{\mathrm{D}}
\newcommand{\rmf}{\mathrm{f}}

\newcommand{\rmG}{\mathrm{G}}
\newcommand{\rmH}{\mathrm{H}}

\newcommand{\rmL}{\mathrm{L}}
\newcommand{\rmM}{\mathrm{M}}

\newcommand{\rmR}{\mathrm{R}}

\newcommand{\rmU}{\mathrm{U}}

\newcommand{\frakg}{\frak{g}}
\newcommand{\frakh}{\frak{h}}

\newcommand{\scrA}{\mathscr{A}}
\newcommand{\scrB}{\mathscr{B}}
\newcommand{\scrC}{\mathscr{C}}

\newcommand{\scrP}{\mathscr{P}}
\newcommand{\scrT}{\mathscr{T}}

\newcommand{\calO}{\mathcal{O}}

\newcommand{\calX}{\mathcal{X}}

\newcommand{\wtx}{\widetilde{x}}

\newcommand{\wtmu}{\widetilde{\mu}}

\newcommand{\ep}{\varepsilon}

\newtheorem{thm}{Theorem}[section]

\newtheorem{coro}[thm]{Corollary}
\newtheorem{defi}[thm]{Definition}

\newtheorem{lem}[thm]{Lemma}
\newtheorem{prop}[thm]{Proposition}

\newtheorem{condition}{Condition}[section]

\def\embds{\hookrightarrow}

\newcommand{\la}{\langle}
\newcommand{\ra}{\rangle}

\DeclareMathOperator{\SL}{SL}
\DeclareMathOperator{\GL}{GL}

\DeclareMathOperator{\diff}{d}

\DeclareMathOperator{\difft}{dt}

\DeclareMathOperator{\Ad}{Ad}
\DeclareMathOperator{\Ave}{Ave}

\DeclareMathOperator{\divisor}{div}
\DeclareMathOperator{\Gal}{Gal}
\DeclareMathOperator{\height}{ht}

\DeclareMathOperator{\LHS}{LHS}

\DeclareMathOperator{\Prob}{Prob}
\DeclareMathOperator{\supp}{supp}

\DeclareMathOperator{\Vol}{Vol}

\newcommand{\norm}[1]{\left\lVert#1\right\rVert}

\date{Nov 2020}

\begin{document}

\title[Counting on homogeneous varieties]{Counting integral points on some homogeneous varieties with large reductive stabilizers}
\author[R.Zhang]{Runlin Zhang}

\email{zhangrunlinmath@outlook.com}
\date{Nov 25, 2020}

\begin{abstract}
Let G be a semisimple group over rational numbers and H be a rational subgroup. Given a rational representation of G and an integral vector x whose stabilizer is equal to H. In this paper we investigate the asymptotic of integral points on Gx with bounded height. We find its asymptotic up to an implicit constant when H is large in G but we allow the presence of intermediate subgroups. This is achieved by a novel combination of two equidistribution results in two different settings: one is that of Eskin, Mozes and Shah on a Lie group modulo a lattice and the other one is a result of Chamber-Loir and Tschinkel on a smooth projective variety with a normal crossing divisor.
\end{abstract}

\maketitle

\tableofcontents

\section{Introduction}

Let $V_0$ be a finite dimensional $\Q$-vector space. Identify $V_0(\R)\cong \R^N$ and hence $V_0(\R)$ is equipped with an Euclidean norm given by
\begin{equation*}
    \norm {(x_1,...,x_N)}:= (x_1^2+...+x_N^2)^{1/2}.
\end{equation*}
Let $B_R:=\{v\in V_0(\R)\,\vert\, \norm{v} \leq R\}$.
We also fix a $\Z$-structure $V_0(\Z)$ on $V_0$.

Let $X \embds V_0$ be an affine variety over $\Q$, a natural question is whether 
$X(\Z) \cap B_R$ is asymptotic to $\Vol(X(\R)\cap B_R)$ for certain natural measure $\Vol$ on $X(\R)$. When $X=V_0$, this is true but it fails in general.
However, one might hope this remains true, at least in some weak sense, in the homogeneous setting.

So let us assume that $X$ is stable and homogeneous (i.e. the action is transitive on $\C$-points) under the linear action of $\bmG$, a closed connected $\Q$-subgroup of $\GL(V_0)$. We assume that $\bmG$ is semisimple. Also assume $X(\Z)$ is non-empty. Then there exists an arithmetic lattice $\Gamma \leq G:=\bmG(\R)$ such that $\Gamma$ preserves the set $X(\Z)$. A theorem of Borel--Harish-Chandra \cite{BorHar62} asserts that $X(\Z)$ decomposes into finitely many $\Gamma$-orbits. Thus it is natural to investigate $ \left\vert
    \Gamma \cdot v_0 \cap B_R
    \right\vert$ for a single $v_0\in X(\Z)$.
The following asymptotic has been established in \cite[Theorem 1.11]{EskMozSha96} (compare \cite{EskMcM93} and \cite{DukRudSar93}) using Ratner's theorem \cite{Rat91} on classifying unipotent-invariant ergodic measure and the linearization technique of Dani--Margulis \cite{DanMar93}.

\begin{thm}
Assume $\bmH\leq \bmG$, the stabilizer group, is maximal among proper connected $\Q$-subgroups and $H:=\bmH(\R)$ has finite volume modulo $H\cap\Gamma$, then 
\begin{equation}\label{bestAsym}
    \lim_{R\to+\infty}
    \frac{ \left\vert
    \Gamma \cdot v_0 \cap B_R
    \right\vert}
    {\mu_{\rmG/\rmH}(\rmG\cdot v_0\cap B_R)} =1.
\end{equation}
\end{thm}

Going beyond this situation, it has been noted in \cite[Section 7]{EskMozSha96} that ``focusing" might happen. As a consequence, they found examples where the above limit exists but may not be equal to $1$. However, they are not able to make a general statement beyond the case of maximal subgroups, though in some special classes of subgroups results of this kind exist (see \cite{GoroTaklTschin15} and \cite{Yang18Rational} for the similar problem of counting rational points). 
The stabilizer in their example has its centralizer not contained in the stabilizer group. If one exclude this situation, then there still might be finitely many intermediate subgroups between $\bmG$ and $\bmH$. For this, the authors of \cite{EskMozSha96} are able the establish this asymptotic \ref{bestAsym} for certain special $(\rho_0, v_0)$ in the presence of intermediate subgroups. For instance they proved 
\begin{thm}
Assume $\bmG=\SL_n$ and $\bmH\leq \bmG$ is a maximal $\Q$-torus that is $\Q$-anisotropic. Let $\rho_0$ be the Adjoint representation $\bmG\to \GL(\frakg)$ and $v_0$ is a rational matrix in $\frakg$ whose centralizer in $\bmG$ is equal to $\bmH$, then asymptotic \ref{bestAsym} holds.
\end{thm}

The purpose of the paper is to establish the weaker asymptotic 
\begin{equation}\label{weakerAsym}
    \text{ there exits }c>0  \text{ such that }\quad
    \lim_{R\to+\infty}
    \frac{ \left\vert
    \Gamma \cdot v_0 \cap B_R
    \right\vert}
    {\mu_{\rmG/\rmH}(G\cdot v_0\cap B_R)} =c
\end{equation}
for certain special $\bmH$ that may not be maximal and for general $(\rho_0,v_0)$.

We shall require $\rmH$ to have finite volume modulo $\rmH\cap \Gamma$ (it is expected that the weaker asymptotic is false without this assumption and we do not go into that direction in this paper, see \cite{OhSha14,KelKon18,ShaZhe18,Zha18,Zha19}) and we denote the probability measure supported on $\rmH\Gamma/\Gamma$ by $\mu_{\rmH}$. Additionally we require the following conditions.

\begin{condition}\label{condition 1}
 $\bmH$ is reductive and $(\bmZ_{\bmG}\bmH)^{\circ}$, the connected component of the centralizer of $\bmH$ in $\bmG$, is contained in $\bmH$.
\end{condition}

\begin{condition}\label{condition 2}
For all connected $\Q$-subgroup $\bmL$ with $\bmH\subset \bmL \subset \bmG$, we have $\rmH$ intersect every connected component of $\rmL$ non-trivially.
\end{condition}

The first condition implies that (see Lemma \ref{lemFiniteInterme})
\begin{itemize}
    \item
     $\bmG$ is semisimple, $\bmH$ is reductive and there are only finitely many intermediate subgroups between $\bmH$ and $\bmG$.
\end{itemize}

The second condition is fulfilled if (see \cite{Matsu64} or \cite[Theorem 14.4]{BorelTits65})
\begin{itemize}
    \item $\bmH$ contains a maximal $\R$-split torus.
\end{itemize}

Therefore, maximal tori containing a maximal $\R$-split torus will meet both conditions.

Now we are ready to state our main theorem.

\begin{thm}\label{thmCounting1}
let $\bmG,\bmH,\Gamma, \rho_0,v_0, \norm{\cdot}$ be the same as above. In addition we assume that $\bmH$ satisfies condition \ref{condition 1} and \ref{condition 2}.  Then the weaker asymptotic \ref{weakerAsym} holds.
\end{thm}

Via a theorem of Borel--Harish-Chandra \cite{BorHar62} and Chamber-Loir--Tschinkel \cite{ChamTschin10}, one can deduce the following.

\begin{coro}\label{CoroCounting2}
Same assumption as above. Fix an integral structure on $\bmV_0$. It induces an integral model $\calX$ of $\bmG\cdot v_0$. Assume $\calX(\Z)\neq \emptyset$. Then there exists a constant $c'>0$ and $a,b\geq 0$ such that 
\begin{equation*}
    \lim_{R\to+\infty} 
    \frac
    {\left\vert
    \calX(\Z)\cap B_R
    \right\vert}
    {c'R^a (\ln{R})^b} =1.
\end{equation*}
When the height function is geometric, then $a=\sigma_{U}$ and $b=b'_{U}$ where $U$ is defined to be union of orbits of $\bmG(\R)$ on $\calX(\R)$ that contain an integral point.
\end{coro}

Being geometric refers to those height functions constructed in Section \ref{secEquiInGeometry}. This may be viewed as an integral version of Manin conjecture with an implicit constant. 
Note that $U$ is not very explicit in general, but when $\rmG$-acts transitively on $\bmG/\bmH(\R)$ (for instance, when the stabilizer group is an $\R$-split torus), one can take $U:=\bmG/\bmH(\R)$ and the constants $a_U$, $b'_U$ agree with the normal one. We do not know whether $a$ being equal to $0$ could actually happen. 

One may also put weights on orbits and ask whether the weighted version of asymptotic \ref{weakerAsym} holds. For this we provide the following equidistribution result, which allows one to put weights given by functions that can be extended to the compactification space.

\begin{thm}\label{thmEquiorCountwithWeights}
Same assumption as above. Then 
\begin{equation}
    \lim_{R\to \infty}\frac{1}{c\mu_{\rmG/\rmH}(B_R)} \sum_{x\in \Gamma\cdot v_0 \cap B_R} \delta_x
\end{equation}
converges to a probability measure on the closure of $\rmG/\rmH$ in $\bbP(V_0\oplus \Q)$. The constant $c$ here is the same as the one in the asymptotic \ref{weakerAsym}.
\end{thm}

Similar statements hold replacing $\bbP(V_0\oplus \Q)$ by $\bbP(V_0)$.
When $\bmH$ is a symmetric subgroup a more precise statement has been established in \cite{GorOhSha09}. In the current situation we do not know if the limit coincides with the following 
\begin{equation*}
     \lim_{R\to \infty}\frac{1}{c\mu_{\rmG/\rmH}(B_R)} \mu_{\rmG/\rmH}\big\vert_{B_R}.
\end{equation*}
When focusing does not arise, this is expected to be true.

Let us list some remarks.

\begin{itemize}
    \item Our results do not provide an explicit expression on $c$. A priori, it depends on all the data.
    In particular methods from \cite{BoroRudn95} or \cite{WeiXu16} does not apply to give a local-to-global type expression for the asymptotic even up to a constant.
    \item
    Our results do not cover the example in \cite[Section 7]{EskMozSha96} as the centralizer of $\bmH$ in that example has non-trivial torus centralizer. However, if the centralizer $\bmZ$ is semisimple, then $\bmG/\bmH \cong \bmG \times \bmZ/ \bmH \times \Delta(\bmZ)$ as a $\bmG$-variety, which would satisfy our condition \ref{condition 1}. Thus in principle, the counting problem can be reduced to the case considered in the present paper (if condition \ref{condition 2} would be satisfied!) if the height function comes from some $(\bmG\times \bmZ)$-equivariant compactification.
    \item
    The height function that can be dealt with must be geometric. We do not know how to handle a general algebraic proper function on $\rmG/\rmH$. For instance, one may change the normal Euclidean norm of the ambient vector space by weights, say, $\left( x_1^2+x_2^6+x_3^9+...  \right)^{1/2}$.
    \item 
    We consider the second condition \ref{condition 2} as a technical matter. Though to remove that condition, one might have to generalize \cite{ChamTschin10} to a setting more general than varieties.
    \item
    It is tempting to apply the same methods to the rational counting problem, as the corresponding results from ergodic theory are available thanks to the work of Gorodnik--Oh \cite{GoroOh11}, assuming $\bmH$ to be semisimple. But an issue similar to condition \ref{condition 2} prevents us from doing so.
    \item
     We do not know whether it is possible for focusing to happen in our set-up. To put it differently, it might be possible that the constant $c$ is equal to $1$.
\end{itemize}

Now let us briefly discuss our approach.

We are going to attack this counting problem via equidistribution results as in \cite{DukRudSar93}. For the approach using height zeta function, the reader is referred to \cite{ChamTschin10,ChamTschin12,TakTsch13}. However rather than showing certain families of measures converges to the full Haar measure on $\rmG/\Gamma$, we shall be content with showing  
\begin{equation*}
    \lim_{R \to \infty}\frac{1}{\mu_{\rmG/\rmH}(B_R)} \int_{B_R} g_*\mu_{\rmH} \,\mu_{\rmG/\rmH}([g])\text{ exists. } 
\end{equation*}

Note that this is really a double integral.
In previous work one usually only take advantage of only one of the integral.

When $\rmH$ is very large (say, maximal or in the setting of \cite{GoroTaklTschin15} and \cite{Yang18Rational}), one choose to show that for generic sequences, the inner integral always converges to the $\rmG$-invariant measure $\mu_{\rmG}$.

In the other extreme, when $\rmH$ is trivial, one choose to study the outer integral using spectral methods or mixing. For example see \cite{Maucou07,GoroNevo12Crelle}. However, their approach only works for certain norms that are stable under perturbation on both sides, a property which does not hold in our set-up.

Once this is done, one has the number of integer points is, up to an implicit constant, asymptotic to the volume, which one can compute via the method of \cite{ChamTschin10}.

The novelty of the present paper is to, in addition to \cite{EskMozSha96}, put into use of the result of \cite{ChamTschin10} to analyze the double integral above.

\subsection*{Outline}

In Section \ref{secBigEquiCount}, we translate the counting problem to an equidistribution problem mentioned above. 

In Section \ref{secBigCompatfy}, we recall the equidistribution result of \cite{EskMozSha96}, using which we define a compactification of $\rmG/\rmH$ inside $\Prob(\rmG/\Gamma)$. 
The idea of compactifying homogeneous spaces in the space of probability measures is not new, for instance it had appeared in Furstenberg's work \cite{Fursten63}, but our usage is very different.
When $\rmH$ satisfies condition \ref{condition 1}, the boundary decomposes into finitely many $\rmG$-orbits. Though this is only a topological object, we show that there is an algebraic one that covers it assuming additionally condition \ref{condition 2}. 

In Section \ref{secBigEquiOnVarie}, we recall the equidistribution result of \cite{ChamTschin10}. To use their theorem, we apply an equivariant resolution of singularities first and then verify that the Haar measure on $\rmG/\rmH$ and the height function from Euclidean norm do arise from the construction of their paper. 

Finally in Section \ref{secWrapUp} we wrap up everything and complete the proof.

\subsection*{Acknowledgement}
We are grateful to useful discussions with Jinpeng An, Osama Khalil, Fei Xu and Pengyu Yang.

\section{Equidistribution and counting}\label{secBigEquiCount}

\subsection{Convention}
Let $\bmG$ be a linear algebraic group over $\Q$ and $\bmH$ be a connected subgroup over $\Q$. 
The Roman letter $\rmH$ is reserved for $\bmH(\R)$.
Let $\Gamma$ be an arithmetic lattice. 
Let $\mu_{\rmH}$ be the Haar measure supported on $\rmH\Gamma/\Gamma$. We always assume $\mu_{\rmH}$ is finite and normalize it to be a probability measure. Similarly, $\mu_{\rmH^{\circ}}$ denotes the probability measure supported on $\rmH^{\circ}\Gamma/\Gamma$. We also let $\wtmu_{\rmH}$ denote the Haar measure of $\rmH$ whose restriction to the fundamental domain of $\rmH/\rmH\cap \Gamma$ coincide with $\mu_{\rmH}$. Then let $\mu_{\rmG/\rmH}$ be the unique $\rmG$-invariant locally finite measure on $\rmG/\rmH$ such that
for all continuous compactly supported function $f$ on $\rmG$, one has
\begin{equation*}
    \int f(x) \wtmu_{\rmG}(x)
    =\int \int f(gh)\wtmu_{\rmH}(h) \,\mu_{\rmG/\rmH}([g]).
\end{equation*}

\subsection{Counting follows from equidistribution}\label{secEquiImplyCount}

Given a continuous transitive $\rmG$-action on a topological space $X$ and $x_0\in X$ such that $g \mapsto g\cdot x_0$ gives a homeomorphism from $\rmG/\rmH$ to $X$. For a proper function $l: X \to [0,\infty)$, we let 
\begin{equation*}
\begin{aligned}
        B_R:=&\left\{
    x \in X\;\vert\; l(x)\leq R
    \right\}.
\end{aligned}
\end{equation*}
As $R$ tends to $\infty$, we ask whether
\begin{equation*}
    \lim_{R\to \infty} \frac{
    |B_R\cap \Gamma\cdot x_0|}{\mu_{\rmG/\rmH}(B_R)} 
\end{equation*}
exists and what it is equal to. By abuse of notation we also identify $B_R$ as a subset in $\rmG/\rmH$ via the orbit map $g\mapsto g\cdot x_0$.

We shall be concerned with situations where $\{B_R\}$ enjoys the following properties:
\begin{itemize}
    \item for any $0<\ep<1$, there exists $U_{\ep}$, a neighborhood of identity in G, such that $B_{(1-\ep)R}\subset gB_R\subset B_{(1+\ep)R}$  for all $R\geq 1$ and $g\in U_{\ep}$;
    \item there exists constants $a,b,c_1,c_2$ with $c_1,c_2>0$ such that 
    \begin{equation*}
        c_1\leq \lim \frac{\mu_{\rmG/\rmH}(B_R)}{R^a(\ln{R})^b}\leq c_2
    \end{equation*}
\end{itemize}
where the $\lim$, when it does not exist, represents a set.
For simplicity, we say that the family $\{B_R\}$ is \textbf{nice} if the above holds. Note that if $\{B_R\}$ is nice, then $\{g\cdot B_R\}$ is also nice for every $g\in\rmG$.
The first condition is equivalent to that
\begin{equation*}
     B_{(1-\ep)R} \subset \bigcap_{g\in U_{\ep}} g B_R \subset
     \bigcup_{g\in U_{\ep}} g B_R \subset B_{(1+\ep)R}.
\end{equation*}
Thus our condition is strictly stronger than the well-roundedness condition in \cite{EskMcM93}, as the second condition is not satisfied when the distance function is induced from the natural Riemannian metric on the associated symmetric space.

\begin{prop}\label{propEquiImplyCount}
Assume $\{B_R\}$ is nice and 
\begin{equation}\label{EquiAveHomoMeas}
\lim_{R\to\infty} \frac{1}{\mu_{\rmG/\rmH}(B_R)} \int_{B_R} g_*\mu_{\rmH} \,\mu_{\rmG/\rmH}([g])
\quad \text{exists}. 
\end{equation}
If we denote the limit measure by $\mu_{\infty}$, then
\begin{itemize}
    \item $\mu_{\infty}$ is absolutely continuous with respect to $\mu_{\rmG}$;
    \item the Radon--Nikodym derivative can be represented by a strictly positive continuous function $f_{\infty}$;
    \item and this function satisfies
    \begin{equation*}
        \lim_{R\to \infty} \frac{
        |B_R\cap g\Gamma\cdot x_0|}{\mu_{\rmG/\rmH}(B_R)} = f_{\infty}([g]).
    \end{equation*}
\end{itemize}
\end{prop}

\begin{proof}
\textit{Step 1, limit exists and is continuous.}\\
We first show that the limit in the third part exists. So take $[g_0]\in \rmG/\Gamma$.

For $0<\ep<1$, choose $U_{\ep}$ satisfying properties as in the definition of being nice. By shrinking to a smaller one we assume $U_{\ep}=U_{\ep}^{-1}$.

Then we choose $V'_{\ep}\subset V_{\ep}$ two families of open neighborhood of identity contained in $U_{\ep}$  such that 
\begin{itemize}
    \item the closure of $V'_{\ep}$ is contained in $V_{\ep}$;
    \item \begin{equation*}
        \lim_{\ep\to 0} 
        \frac{\mu_{\infty}(V_{\ep}'[g_0])}
        {\mu_{\infty}(V_{\ep}[g_0])} =
        \lim_{\ep\to 0} 
        \frac{\mu_{\rmG}(V_{\ep}'[g_0])}
        {\mu_{\rmG}(V_{\ep}[g_0])}
        =1.
    \end{equation*}
\end{itemize}
Then we choose a continuous function $f_{\ep}$ with 
$1_{V'_{\ep}[g_0]}\leq f_{\ep} \leq 1_{V_{\ep}[g_0]}$.
    
    We write
    \[
    \Phi_R([g]):=\frac{|B_R\cap g\Gamma \cdot x_0|}{\mu_{\rmG/\rmH}(B_R)} = \frac{1}{\mu_{\rmG/\rmH}(B_R)}\sum_{\gamma \in \Gamma/\Gamma\cap \rmH} 1_{B_R}(g\gamma\cdot x_0).
    \]
    Then we have
    \begin{equation}\label{EquaProofProp2.2}
    \begin{aligned}
         \limsup_R  \la \Phi_R, 1_{V_{\ep}'[g_0]} \ra_{\mu_{\rmG}} 
        &\leq 
        \lim_R  \la \Phi_R, f_{\ep} \ra_{\mu_{\rmG}}  \\
        &=  \int f_{\ep}(x) \mu_{\infty}(x)
        \leq \liminf_R \la \Phi_R, 1_{V_{\ep}[g_0]} \ra_{\mu_{\rmG}};\\
        \mu_{\infty} (V_{\ep}'[g_0])
        &\leq 
        \int f_{\ep}(x) \mu_{\infty}(x)
        \leq 
        \mu_{\infty} (V_{\ep}[g_0])
    \end{aligned}
    \end{equation}
    where the equality on the first line comes from an unfolding argument (see \cite{DukRudSar93, EskMcM93}).
    Now we can start to estimate. Assume $\ep$ to be small enough so that the natural map $U_{\ep}\to U_{\ep}[g_0]$ is a homeomorphism. Then $\wtmu_{\rmG}\vert_{U_\ep}$ is identified with $\mu_{\rmG}\vert_{U_{\ep}[g_0]}$ under this homeomorphism.
    \begin{equation*}
        \begin{aligned}
                \la \Phi_R, 1_{V_{\ep}'[g_0]} \ra_{\mu_{\rmG}} =& 
                \frac{1}{\mu_{\rmG/\rmH}(B_R)} \int _{[g]\in V_{\ep}'[g_0]} |B_R\cap g\Gamma \cdot x_0| \,\mu_{\rmG}([g])\\
                =&\frac{1}{\mu_{\rmG/\rmH}(B_R)} \int _{g\in V_{\ep}'} 
                |g^{-1}B_R \cap g_0\Gamma\cdot x_0| \,\wtmu_{\rmG}(g)
                \\
                \geq&
               \frac{\mu_{\rmG}(V_{\ep}'[g_0])}
               {\mu_{\rmG/\rmH}(B_R)} 
               \Big\vert 
               (\bigcap_{g\in U_{\ep}} gB_R) \cap 
               g_0\Gamma\cdot x_0 \Big\vert
                \\
                \geq& \frac{\mu_{\rmG}(V_{\ep}'[g_0])}{\mu_{\rmG/\rmH}(B_R)} 
                \left\vert B_{(1-\ep)R}\cap g_0\Gamma\cdot x_0 \right \vert
        \end{aligned}
    \end{equation*}
    
    By taking the $\limsup$ as $R$ tends to $\infty$ and using Equation \ref{EquaProofProp2.2} above we get 
    \begin{equation*}
        \mu_{\rmG}(V_{\ep}'[g_0])\limsup 
        \frac{\left\vert  B_{(1-\ep)R} \cap g_0\Gamma\cdot x_0 \right\vert} {\mu_{\rmG/\rmH}(B_R)}
        \leq \limsup \la \Phi_R, 1_{V_{\ep}'[g_0]} \ra_{\mu_{\rmG} }
        \leq  \mu_{\infty} (V_{\ep}[g_0]).
    \end{equation*}
    Similarly we have 
    \begin{equation*}
         \mu_{\rmG}(V_{\ep}[g_0])\liminf 
        \frac{\left\vert  B_{(1+\ep)R} \cap g_0\Gamma\cdot x_0 \right\vert}
        {\mu_{\rmG/\rmH}(B_R)}
        \geq  \mu_{\infty} (V'_{\ep}[g_0]).
    \end{equation*}
    Combining these two while replacing $(1-\ep)R$ by $R$ in the first inequality and $(1+\ep)R$ by $R$ in the second inequality, we get
    \begin{equation*}
    \begin{aligned}
            & \frac{\mu_{\infty}(V'_{\ep}[g_0])}{\mu_{\rmG}(V_{\ep}[g_0])} \liminf 
         \frac{\mu_{\rmG/\rmH}(B_{R/(1+\ep)})}{\mu_{\rmG/\rmH}(B_{R})}
         \leq
        \liminf 
        \frac{ \left\vert B_{R}\cap g_0\Gamma\cdot x_0 \right\vert}{\mu_{\rmG/\rmH}(B_{R})} \\
        \leq &
        \limsup 
        \frac{ \left\vert B_{R}\cap g_0\Gamma\cdot x_0 \right\vert}{\mu_{\rmG/\rmH}(B_{R})}
        \leq 
        \frac{\mu_{\infty}(V_{\ep}[g_0])}{\mu_{\rmG}(V'_{\ep}[g_0])}
        \limsup 
         \frac{\mu_{\rmG/\rmH}(B_{R/(1-\ep)})}{\mu_{\rmG/\rmH}(B_{R})}.
    \end{aligned}
    \end{equation*}
    Taking  $\limsup_{\ep\to 0}$ on the left end and recall the definition of being nice, we have
    \begin{equation*}
    \begin{aligned}
          &\limsup_{\ep\to0} \liminf_{R\to \infty} \frac{\mu_{\rmG/\rmH}(B_{R/(1+\ep)})}{\mu_{\rmG/\rmH}(B_R)}=1
        \\
                \implies& \limsup_{\ep\to 0} 
                \frac{\mu_{\infty}(V'_{\ep}[g_0])}{\mu_{\infty}(V_{\ep}[g_0])} \cdot
        \frac{\mu_{\infty}(V_{\ep}[g_0])}{\mu_{\rmG}(V'_{\ep}[g_0])} \cdot  
         \frac{\mu_{\rmG}(V'_{\ep}[g_0])}{\mu_{\rmG}(V_{\ep}[g_0])} =
         \limsup_{\ep \to 0} \frac{\mu_{\infty}(V_{\ep}[g_0])}{\mu_{\rmG}(V'_{\ep}[g_0])}
    \end{aligned}
    \end{equation*}
    which is equal to the $\limsup_{\ep \to 0}$ of the right end.
    Thus 
    \begin{equation*}
         \liminf 
        \frac{ \left\vert B_{R}\cap g_0\Gamma\cdot x_0 \right\vert}{\mu_{\rmG/\rmH}(B_{R})} 
        =
        \limsup 
        \frac{ \left\vert B_{R}\cap g_0\Gamma\cdot x_0 \right\vert}{\mu_{\rmG/\rmH}(B_{R})} \implies
          \lim 
        \frac{ \left\vert B_{R}\cap g_0\Gamma\cdot x_0 \right\vert}{\mu_{\rmG/\rmH}(B_{R})} \text{ exists}. 
    \end{equation*}
    Let us note that the above limit is equal to
    \begin{equation*}
         \lim_{\ep \to 0} \frac{\mu_{\infty}(V_{\ep}[g_0])}{\mu_{\rmG}(V_{\ep}[g_0])} 
         =   \lim_{\ep \to 0} \frac{\mu_{\infty}(V'_{\ep}[g_0])}{\mu_{\rmG}(V'_{\ep}[g_0])}.
    \end{equation*}
    Call this limit $f([g_0])$. 
    Now we have a function $f: \rmG/\Gamma \to \R$, which is continuous because for $u\in U_{\ep}$,
    \begin{equation*}
    \begin{aligned}
        f(u[g_0])= & \lim 
        \frac{ \left\vert B_{R}\cap ug_0\Gamma\cdot x_0 \right\vert}{\mu_{\rmG/\rmH}(B_{R})} =
         \lim 
        \frac{ \left\vert u^{-1}B_{R}\cap g_0\Gamma\cdot x_0 \right\vert}{\mu_{\rmG/\rmH}(B_{R})} \\
        \geq& \limsup  
        \frac{ \left\vert B_{R(1-\ep)}\cap g_0\Gamma\cdot x_0 \right\vert}{\mu_{\rmG/\rmH}(B_{R})} 
        = \limsup  
        \frac{ \left\vert B_{R}\cap g_0\Gamma\cdot x_0 \right\vert}{\mu_{\rmG/\rmH}(B_{R/(1-\ep)})},
    \end{aligned}
    \end{equation*}
    converging to $f[g_0]$ as $\ep \to 0$. A similar argument would show the limit is no larger than $f[g_0]$.
    Repeating the same argument it is not hard to see that for any two $g_1,g_0$ in $\rmG$, there exists a positive constant $c_{g_1,g_0}$ such that $f([g_1g_0])\geq c_{g_1,g_0}f([g_0])$. Therefore $f$ is either constantly equal to $0$ or strictly positive.
    \\
    
    \textit{Step 2, absolute continuity.}\\
    Now we switch our attention to $\mu_{\infty}$ and show that it is absolutely continuous with respect to $\mu_{\rmG}$.
    That is to say, we need to show $\mu_{\infty}(E)=0$ whenever $\mu_{\rmG}(E)=0$.

    Recall that $\wtmu_{\rmG}$ is the Haar measure on $\rmG$ whose restriction to a fundamental set gives $\mu_{\rmG}$.
    We first observe that if a set $E$ has measure $0$ with respect to $\mu_{\rmG}$, then for any probability measure $\lambda$ and any bounded nonempty open set $U$ in $\rmG$, we have
    \begin{equation*}
        \left(\int_{u\in U} u_*\lambda \,\wtmu_{\rmG}(u)
        \right)
        (E)=0.
    \end{equation*} Indeed by decomposing $E$ into smaller pieces and shrinking $U$, we assume that $U=U^{-1}$ and $E$ is the image of some $\widetilde{E}$ under the natural projection $\rmG\to \rmG/\Gamma$ and that the projection is injective when restricted to $U^{-1}\widetilde{E}$ . Using this bijection (onto $U^{-1}E$) we lift $\lambda\vert_{U^{-1}E}$ to $\widetilde{\lambda}$ supported on $U^{-1}\widetilde{E}$. Moreover, we require that for each $e \in E$, the map from $U$ to $\rmG/\Gamma$ via $u\mapsto ue$ is injective.

    Now by interchanging the order of integration,
    \begin{equation*}
    \begin{aligned}
           \int_{u\in U}  u_*\lambda(E)\, \wtmu_{\rmG}(u)
           &= 
           \int_{u\in U}\int_{\rmG/\Gamma}1_{u^{-1}E}(x)  \,\lambda(x) \wtmu_{\rmG}(u)\\
           &=
           \int_{u\in U}  \int_{\rmG}  1_{u^{-1}E}(x)  \,\widetilde{\lambda}(x) \wtmu_{\rmG}(u)\\
           &=\int\int 1_{\{(u,x)\in U\times \widetilde{E}  \,\vert \, ux\in \widetilde{E}  \}} (u,x) \,\wtmu_{\rmG}(u)\widetilde{\lambda}(x)\\
           &=\int  \wtmu_{\rmG}(\widetilde{E}x^{-1}) \,\widetilde{\lambda}(x) =0,
    \end{aligned}
    \end{equation*}
    which confirms our observation.
    
    On the other hand, for an $E$ with $\mu_{\rmG}(E)=0$, we can find a family of shrinking measurable sets $(E_i)$ such that $E=\cap E_i$ and $\mu_{\infty}(\partial E_i)=0$. For $u\in \rmG$,
    \begin{equation*}
    \begin{aligned}
         \mu_{\infty}(E_i) &= \lim_R \frac{1}{\mu_{\rmG/\rmH}(B_R)}\int_{[g]\in B_R}g_*\mu_{\rmH} (E_i)\,\mu_{\rmG/\rmH}([g]),\\
          u_*\mu_{\infty}(E_i) &= \lim_R \frac{1}{\mu_{\rmG/\rmH}(B_R)}\int_{[g]\in u\cdot B_R}g_*\mu_{\rmH} (E_i)\,\mu_{\rmG/\rmH}([g]).
    \end{aligned}
    \end{equation*}
    The following difference can be estimated by 
    \begin{equation*}
    \begin{aligned}
        &\left| \mu_{\infty}(E_i)- \frac{1}{\wtmu_{\rmG}(U_{\ep})}\int_{U_{\ep}} u_*\mu_{\infty}(E_i) \,\wtmu_{\rmG}(u) \right|\\
        \leq & \lim_{R\to\infty} \frac{1}{\wtmu_{\rmG}(U_{\ep})}\int_{ U_{\ep}}
        \frac{1}{\mu_{\rmG/\rmH}(B_R)} \int_{uB_R\Delta B_R} g_*\mu_{\rmH}(E_i) \,\mu_{\rmG/\rmH}([g])\wtmu_{\rmG}(u)\\
        \leq & \lim_{R\to\infty} \frac{1}{\wtmu_{\rmG}(U_{\ep})}
        \int_{ U_{\ep}}
        \frac{1}{\mu_{\rmG/\rmH}(B_R)} \int_{B_{(1+\ep)R}\setminus B_{(1-\ep)R}} g_*\mu_{\rmH}(E_i) \,\mu_{\rmG/\rmH}([g])\wtmu_{\rmG}(u)\\
        \leq & 
        \lim_{R\to\infty} \frac{\mu_{\rmG/\rmH}(B_{(1+\ep)R}\setminus B_{(1-\ep)R})}
        {\mu_{\rmG/\rmH}(B_R)} \leq \ep'
    \end{aligned}
    \end{equation*}
    for some $\ep'$, which is independent of $i$ and converges to $0$ as $\ep$ does so. So if we let $i$ tends to infinity, we get
    \begin{equation*}
         \left|\mu_{\infty}(E)- \frac{1}{\mu_{\rmG}(U_{\ep})}\int_{u\in U_{\ep}} u_*\mu_{\infty}(E) \wtmu_{\rmG}(u) \right| \leq \ep'.
    \end{equation*}
    However, by our observation, the measure of $E$ against ``Haar average'' of any probability measure is $0$. So we are left with
    \begin{equation*}
         \mu_{\infty}(E) \leq \ep'
    \end{equation*}
    for $\ep'$ arbitrarily small, which forces $\mu_{\infty}(E) =0$. This ends the proof of absolute continuity.
    
    \textit{Step 3, completing the proof.}\\
    Now we write $\mu_{\infty}= \psi \cdot \mu_{\rmG}$ for some non-negative function $\psi$ in $L^1(\mu_{\rmG})$.
    
    By what has been shown in Step 1, for any $[g_0]\in \rmG/\Gamma$,
    \begin{equation*}
        f([g_0])=
        \lim_{\ep \to 0}
        \frac{\mu_{\infty}(V_{\ep}[g_0])}{\mu_{\rmG}(V_{\ep}[g_0])}
        =
        \lim_{\ep \to 0}
        \frac{1}{\wtmu_{\rmG}(V_{\ep})} \int_{u\in V_{\ep}} \psi(u[g_0]) \wtmu_{\rmG}(u)
    \end{equation*}
    Therefore $f=\psi$ almost surely. As $\psi \mu_{\rmG}$ is a probability measure, we see that $f$ can not be the $0$ function. As said in step 1, this implies that $f$ is strictly positive.
\end{proof} 

\subsection{A reformulation}\label{secAReformulation}

We reformulate the required equidistribution assumption \ref{EquiAveHomoMeas} from Proposition \ref{propEquiImplyCount} in a form that will be actually proved later.

Let $\text{Ave}$ be the continuous map $\Prob(\Prob(\rmG/\Gamma)) \to \Prob(\rmG/\Gamma)$ defined by
\begin{equation*}
    \int f(x) \Ave(\lambda)(x) := 
    \int \left(\int f(x)\mu(x) \right) \lambda(\mu)
    \quad
    \forall f \in C_c(\rmG/\Gamma)
\end{equation*}
where both spaces are equipped with the weak-$*$ topology. Note that this makes sense as the map $\mu \mapsto \int f(x)\mu(x)$ is a bounded continuous map on $\Prob(\rmG/\Gamma)$ if $f$ is continuous and bounded. For a point $x \in \Prob(\rmG/\Gamma)$, we let $\delta_x$ be the probability measure supported on $\{x\}$.

\begin{prop}\label{propEquionProb}
In the set-up of Theorem \ref{thmCounting1}, the limit of the following
\begin{equation*}
    \frac{1}{\mu_{\rmG/\rmH}(B_R)}\int_{[g]\in B_R} g \delta_{\mu_{\rmH}} \,\mu_{\rmG/\rmH}([g])
\end{equation*}
exists in $\Prob(\Prob(\rmG/\Gamma))$.
\end{prop}

\begin{proof}[Proof of Theorem \ref{thmCounting1} assuming Proposition \ref{propEquionProb}]
By Proposition \ref{propEquiImplyCount}, it suffices to show that the Equation \ref{EquiAveHomoMeas} holds. But by Proposition \ref{propEquionProb} and the fact that $\Ave$ is continuous we have that 
\begin{equation*}
    \lim_{R\to \infty}
    \Ave
    \left(
     \frac{1}{\mu_{\rmG/\rmH}(B_R)}\int_{[g]\in B_R} g \delta_{\mu_{\rmH}} \,\mu_{\rmG/\rmH}([g])
    \right) \text{ exists}
\end{equation*}
which is nothing but Equation \ref{EquiAveHomoMeas}.
\end{proof}

The proof of Proposition \ref{propEquionProb} will be delegated to Section \ref{secWrapUp}.

\section{Compactifications}\label{secBigCompatfy}

\subsection{Convention}
Throughout this section we make the following assumption unless otherwise noted.
Take $\bmH\leq \bmG$ to be connected reductive groups defined over $\Q$ and assume that $\bmH$ satisfies both condition \ref{condition 1} and \ref{condition 2}. Let $\Gamma$ be an arithmetic lattice in $\rmG$ and assume $\rmH\cap \Gamma$ is a lattice in $\rmH$.

\subsection{Translates of homogeneous measures}
We need the following three inputs from the work of \cite{EskMozSha96, EskMozSha97}.

\begin{thm}\label{thmEMSequi1}
For any sequence $(g_n)$ in $\rmG$, there exists a subsequence $(g_{n_k})$, a connected $\Q$-subgroup $\bmL$ of $\bmG$ with no nontrivial $\Q$-characters, a bounded sequence $(\delta_n )$ in $\rmG$, $(\gamma_n)$ in $\Gamma$ and $(h_n)$ in $\rmH^{\circ}$ such that
\begin{enumerate}
    \item $g_n =\delta_n \gamma_n h_n$;
    \item $\lim_k \delta_{n_k}$ exists, which we denote by $\delta_{\infty}$;
    \item $\gamma_{n_k} \bmH \gamma_{n_k}^{-1} \subset \bmL$ for all $k$;
    \item $\lim_k (g_{n_k})_*\mu_{\rmH^{\circ}} = \delta_{\infty}\mu_{\rmL^{\circ}}$.
\end{enumerate}
\end{thm}

\begin{thm}\label{thmEMSequi2}
If $(\lambda_n)$ is a sequence in $\Gamma$ and $\bmL$ is a connected $\Q$-subgroup of $\bmG$ satisfying that
\begin{itemize}
    \item $\lambda_n \bmH \lambda_n^{-1}$ is contained in $\bmL$  for all $n$ and every proper connected $\Q$-subgroup $\bmL'$ of $\bmL$ only contains $\lambda_n \bmH \lambda_n^{-1}$ for at most finitely many $n$,
\end{itemize}
then $\lim_n \lambda_n \mu_{\rmH^{\circ}}= \mu_{\rmL^{\circ}}$.
\end{thm}

The following, \cite[Lemma 5.1, Lemma 5.2]{EskMozSha96}, is the algebra behind the focusing criterion in \cite[Corollary 1.15]{EskMozSha96}.

\begin{thm}\label{thmEMSfocusing}
Let $\bmM$ and $\bmL$ be two reductive $\Q$-subgroups of $\bmG$. Let $\bmX(\bmM,\bmL):=\{g\in \bmG\,\vert\, g\bmM g^{-1} \subset \bmL\}$. Then there is a finite set $D\subset\bmX(\bmM,\bmL)\cap \Gamma $ such that 
$\bmX(\bmM,\bmL)\cap \Gamma  = \rmL \cap \Gamma   \cdot D \cdot (\bmZ_{\bmG}\bmM)^{\circ} \cap \Gamma$.
Also, there exists a finite set $D'\subset \bmX(\bmM,\bmL)$
such that $\bmX(\bmM,\bmL)=\bmL \cdot D'\cdot (\bmZ_{\bmG}\bmM)^{\circ}$.
\end{thm}

We have used the same notation for a variety and its complex points and $\bmZ_{\bmG}\bmM$ means the centralizer of $\bmM$ in $\bmG$.

\begin{thm}\label{thmEMSnondiverg1}
Let $\bmL \leq \bmG$ be a reductive subgroup over $\Q$ and assume that $\bmZ_{\bmG}\bmL/\bmZ(\bmL)$ is $\Q$-anisotropic. Then there exists a bounded set $B\subset \rmG$ such that $\rmG=B\cdot \Gamma\cdot \rmL$.
\end{thm}
When $\bmZ_{\bmG}\bmL$ is $\Q$-anisotropic, this is \cite[Theorem 1.3]{EskMozSha97}. In general this follows from \cite[Theorem 1.7]{Zha19}.

As we are going to concern with translates of $\mu_{\rmH}$ rather than $\mu_{\rmH^{\circ}}$, let us prove the following corollary.

\begin{coro}\label{coroTranFullComponents}
Same notation as in Theorem 3.1, 3.2. Moreover we have that 
\begin{itemize}
    \item $\lim_k g_{n_k} \mu_{\rmH} = \delta_{\infty}\mu_{\rmL}$;
    \item $\lim_n \lambda_n \mu_{\rmH} = \mu_{\rmL}$.
\end{itemize}
\end{coro}

\begin{proof}
It suffices to show the second one.

By assumption, $x \mapsto \lambda_n x \lambda_n^{-1}$ defines a homomorphism from $\rmH$ to $\rmL$, which maps $\rmH^{\circ}$ to $\rmL^{\circ}$. Thus it descends to a homomorphism $\bmc_{\lambda_n}: \rmH/\rmH^{\circ} \to \rmL/\rmL^{\circ}$. By our assumption \ref{condition 2}, this is surjective for all $n$. Also for $y \in \rmL/\rmL^{\circ}$, the notation $y\rmL^{\circ}\Gamma/\Gamma$ or $y\mu_{\rmL^{\circ}}$ makes sense.

Now we claim that fixing $y_0 \in \rmL/\rmL^{\circ}$, for any $x_n\in \bmc_{\lambda_n}^{-1}(y_0)$,
\subsubsection{Claim} $\lim_n \lambda_n x_n \mu_{\rmH^{\circ}} =y_0\mu_{\rmL^{\circ}}$.

\begin{proof}[Proof of the claim]
It suffices to show that for any infinite subsequence, there exists a further subsequence where the limit exists and is the expected one. Therefore we may feel free to pass to a further subsequence whenever needed.

We may assume that, by passing to a subsequence, 
$\lambda_n x_n =\delta_n'\gamma_n'h_n'$ with $\delta'_n$ converging to $\delta'$, $\gamma_n'\in \Gamma$ and $h_n'\in \rmH^{\circ}$. Moreover, there exists a connected $\Q$-subgroup $\bmM$ that contains $\gamma_n'\bmH\gamma_n'^{-1}$ for all $n$ and every proper connected $\Q$-subgroup of $\bmM$ contains at most finitely many  $\gamma_n'\bmH\gamma_n'^{-1}$. Hence $\lambda_n x_n \mu_{\rmH^{\circ}}$ converges to $\delta'\mu_{\rmM^{\circ}}$ by Theorem \ref{thmEMSequi2}.

On the other hand, as $\lambda_n x_n \rmH^{\circ} \lambda_n^{-1 }$ is contained in $y_0\rmL^{\circ}$, we have $\lambda_n x_n \rmH^{\circ}\Gamma \subset y_0\rmL^{\circ}\Gamma$ for all $n$. Thus $\delta'{\rmL^{\circ}}\Gamma$ is contained in $y_0\rmL^{\circ}\Gamma$. By the Lemma \ref{lemOneObtContainAnother} below, there exists $l_{\rmM}\in \rmL^{\circ}$ and $\gamma_{\rmM} \in \Gamma$ such that $\delta'=y_0l_{\rmM}\gamma_{\rmM}$ and $\gamma_{\rmM} \rmM^{\circ} \gamma_{\rmM}^{-1} \subset \rmL^{\circ}$. Hence 
\begin{equation*}
    \lambda_n x_n = (\delta_n'\delta'^{-1}y_0l_{\rmM})\cdot (\gamma_{\rmM}\gamma_n') \cdot (h_n').
\end{equation*}
Therefore, by replacing $\delta_n'$ with $\delta_n'\delta'^{-1}y_0l_{\rmM}$, $\gamma_n'$ with $\gamma_{\rmM}\gamma_n'$ and $\bmM$ with $\gamma_{\rmM} \bmM \gamma_{\rmM}^{-1}$ we may assume that 
\begin{itemize}
    \item $\delta_n'$ converges to $y_0l_{\rmM}$ for some $l_{\rmM} \in \rmL^{\circ}$;
    \item $\gamma_n'\bmH \gamma_n'^{-1} \subset \bmM \subset \bmL$ for all $n$.
\end{itemize}

By switching the role played by $\gamma_n'$ and $\lambda_n$, we sees that $\dim \bmM = \dim \bmL$ and so $\bmM=\bmL$. Hence the proof completes.

\end{proof}

Now we continue to prove the corollary.

\begin{equation*}
    \begin{aligned}
           \lambda_n \mu_{\rmH} 
           = \frac{1}{|\rmH/\rmH^{\circ}|} \sum_{x\in \rmH/\rmH^{\circ}} \lambda_n x\mu_{\rmH^{\circ}} 
           =  \frac{1}{|\rmH/\rmH^{\circ}|} \sum_{y\in \rmL/\rmL^{\circ}}\sum_{x\in c_{\lambda_n}^{-1}(y)} \lambda_n x\mu_{\rmH}. 
    \end{aligned}
\end{equation*}
By our claim for any $x_n \in c_{\lambda_n}^{-1}(y)$, $\lambda_n x_n \mu_{\rmH^{\circ}}$ converges to $y\mu_{\rmL^{\circ}}$. Therefore
\begin{equation*}
    \lim_{n\to \infty} \lambda_n \mu_{\rmH} = 
     \frac{1}{|\rmH/\rmH^{\circ}|} \sum_{y\in \rmL/\rmL^{\circ}} y \mu_{\rmL^{\circ}} = \mu_{\rmL}.
\end{equation*}
\end{proof}

\begin{lem}\label{lemOneObtContainAnother}
Let $\bmA, \bmB$ be two connected $\Q$-subgroup of $\bmG$ such that $A^{\circ}\Gamma$ and $B^{\circ}\Gamma$ are closed. Let $g,h\in \rmG$. Then $gA^{\circ}\Gamma \subset hB^{\circ}\Gamma$ implies that there exist $b\in B^{\circ}$ and $\gamma\in\Gamma$ such that
\begin{itemize}
    \item $g=hb\gamma$;
    \item $\gamma A \gamma^{-1}\subset B$.
\end{itemize}
\end{lem}

\begin{proof}
By assumption there are $b\in B^{\circ}$, $\gamma\in \Gamma$ such that $g=hb\gamma$.
So we have
\begin{equation*}
    hb\gamma A^{\circ}\Gamma \subset hB^{\circ}\Gamma
    \text{ which is equivalent to } \gamma A^{\circ} \gamma^{-1} \subset \bigcup_{k \in \Gamma} B^{\circ} k.
\end{equation*}
By Baire Category theorem there exists $k\in \Gamma$ such that $B^{\circ} k $ contains a nonempty open set in $A^{\circ}$, which is Zariski dense in $A^{\circ}$. Hence we must have $\gamma A^{\circ} \gamma^{-1} \subset B^{\circ} k$ for this $k \in \Gamma$. By observing that the left hand side contains the identity, we have $B^{\circ}k=B^{\circ}$. It then follows that $\gamma A\gamma^{-1}\subset B$.
\end{proof}

\subsection{Compactification in the space of measures}

\begin{defi}
Let $\Prob(\rmG/\Gamma)$ be the space of probability measures on $\rmG/\Gamma$ equipped with the weak-$*$ topology. We define
\begin{equation*}
    \begin{aligned}
            \iota_0 : \rmG/\rmH &\longrightarrow \Prob(\rmG/\Gamma) \\
            [g] &\longmapsto g_*\mu_{\rmH}
    \end{aligned}
\end{equation*}
and let $X_0:=\overline{\iota_0(\rmG/\rmH)}$ and $D_0:= X_0\setminus \iota_0(\rmG/\rmH)$.
\end{defi}

Note that $X_0$ is compact by Theorem \ref{thmEMSequi1} (the reader is reminded that since we are dealing with a possibly noncompact space, the space of probability measures may not be compact). This is always true as long as the centralizer of $\bmH$ in $\bmG$ is $\Q$-anisotropic and both $\bmH$, $\bmG$ are reductive.

\begin{defi}
\begin{equation*}
    \scrA:= \{\bmL\leq_{\Q} \bmG\,\vert\, \bmL \text{ connected and } \exists \text{ a sequence } (\gamma_n) \text{ in }\Gamma \text{ such that } \gamma_n \mu_{\rmH} \to \mu_{\rmL} 
    \}
\end{equation*}
and we define an equivalence relation on $\scrA$ by
\begin{equation*}
    \bmL \sim \bmL' \iff \exists \gamma \in \Gamma,\, \gamma \bmL \gamma^{-1} = \bmL'.
\end{equation*}
\end{defi}

The set $\scrA/{\sim}$ is finite due to the next lemma.
\begin{lem}\label{lemFiniteInterme}
Assume that $\bmL$ is a connected subgroup of some connected reductive $\Q$-group $\bmM$ and satisfies condition \ref{condition 1}, then there are only finitely many intermediate reductive groups between $\bmL$ and $\bmM$. Everything is over $\C$ here.
\end{lem}

\begin{proof}
First we claim that up to $\bmG(\C)$-conjugacy, there are only finitely many subgroups (over $\C$) $\bmA$ satisfying condition \ref{condition 1} that is isomorphic to some fixed linear algebraic group $\bmA_0$. Indeed for each $\bmA$, we write $\bmA=\bmZ\cdot \bmS$ where $\bmZ$ is the connected center and $\bmS$ is a semisimple group. 
By \cite[Lemma A.1]{EinMarVen09}, there are only finitely many possibilities for $\bmS$ up to $\bmG(\C)$-conjugacy. So suppose $\bmA_i=\bmZ_i\cdot\bmS_i$ for $i=1,2$ are two groups isomorphic to $\bmA$ and $g\in \bmG(\C)$ is such that $g\bmS_1g^{-1}=\bmS_2$. 
Hence $g\bmZ_1g^{-1}\subset \bmZ_{\bmG}\bmS_2$. By condition \ref{condition 1}, we have that $g\bmZ_1g^{-1}$ and $\bmZ_2$ are both maximal tori in $\bmZ_{\bmG}{\bmS_2}$. So $hg\bmZ_1g^{-1}h^{-1}=\bmZ_2$ for some $h\in \bmZ_{\bmG}\bmS_2(\C)$. Hence $hg\bmA_1g^{-1}h^{-1}=\bmA_2$. This proves the claim.

Now fix $\bmA$ that is connected, reductive and contains $\bmL$. It suffices to show that the collection (denoted by $\scrB$) of $\bmB$ that is conjugate to $\bmA$ and contains $\bmL$ is finite. 
By Theorem \ref{thmEMSfocusing}, there exists a finite set $D\subset \bmX(\bmL,\bmA)$ such that $\bmX(\bmL,\bmA)=\bmA \cdot D \cdot (\bmZ_{\bmM}\bmL)^{\circ}$. 
For each $\bmB\in\scrB$, find $g$ such that $g\bmB g^{-1}=\bmA$. Then $g\in \bmX(\bmL,\bmA)$. So $g=adz$ for some $a\in \bmA$, $d\in D$ and $z\in (\bmZ_{\bmM}\bmL)^{\circ}$, which is contains in $\bmL$ by assumption. So $z$ is also contained in $d^{-1}\bmA d$. Hence $\bmB =  g^{-1}\bmA g = z^{-1} d^{-1} \bmA d z= d^{-1}\bmA d$. This finishes the proof.

\end{proof}

We can describe $X_0$ rather explicitly thanks to Corollary \ref{coroTranFullComponents} and Lemma \ref{lemOneObtContainAnother}. And Lemma \ref{lemFiniteInterme} above says that the disjoint union is finite.
\begin{coro}\label{coroBoundStrataProb}
As a set $X_0= \bigsqcup_{[\bmL]\in \scrA/\sim} \rmG\cdot \mu_{\rmL}$.
\end{coro}

\subsection{An algebraic compactification}\label{cpctfyalgebraic}
For each $\bmL\in \scrA$ that contains $\bmH$,  fix a representation $(\rho_{\bmL}, \bmV_{\bmL})$ and a $\Q$-vector $v_{\bmL}\in \bmV_{\bmL}$ whose stabilizer in $\bmG$ is exactly equal to $\bmL$. So we have fixed finitely many vectors by Lemma \ref{lemFiniteInterme}. For a finite dimensional linear space $\bmV$, we let $\bbP(\bmV)$ be the associated projective variety.

\begin{defi}\label{defiComptfyX_1}
\begin{equation*}
    \begin{aligned}
            \iota_1:\, &\rmG/\rmH \to \bigoplus\bmV_{\bmL} \to \bigoplus\bbP(\bmV_{\bmL}\oplus \Q) \\
            &g \mapsto \oplus g\cdot \bmv_{L},\, \oplus v \mapsto \oplus [v:1]
    \end{aligned}
\end{equation*}
We define $X_1:= \overline{\iota_1(\rmG/\rmH)}$ and $\bmX_1:= \overline{\overline{\iota_1(\bmG/\bmH) }}$, a projective variety over $\Q$.
\end{defi}
Here and elsewhere we use $\overline{\text{something}}$ to denote the closure in Hausdorff topology while $\overline{\overline{\text{something}}}$ is reserved for Zariski topology.

By extending the action of $\bmG$ to $(\rho_{\bmL}\oplus 1, \bmV_{\bmL}\oplus \Q )$ . We see that $\iota_1$ is $\bmG$-equivariant.

\begin{defi}\label{defi}
Given a sequence $(g_n)$ in $\rmG$ such that $\iota_1([g_n])$ converges. We let $\scrA_0(g_n)$ be the collection of  $\bmL$  such that $( g_n v_{\bmL})$ converges in $V_{\bmL}$.
\end{defi}

\begin{lem}\label{lemMinGpStableUnder_g_n}
If $\bmL_1, \bmL_2 \in \scrA_0(g_n)$, then $(\bmL_1\cap \bmL_2)^{\circ} \in \scrA_0(g_n)$. Therefore there exists a unique minimal element in $\scrA_0(g_n)$.
\end{lem}

\begin{proof}
By assumption $\lim g_n v_{L_i}$ (i=1,2) exists in $\bmV_{\bmL_i}$ (i=1,2), therefore $\lim g_n(v_{\bmL_1}\oplus v_{\bmL_2})$ exists in $\bmV_{\bmL_1}\oplus \bmV_{\bmL_2}$. Hence $\lim [g_n]$ exists in $\bmG/\bmL_1\cap \bmL_2$ and so $[g_n]$ is bounded in $\bmG/(\bmL_1\cap \bmL_2)^{\circ}$.
So $g_n v_{(\bmL_1\cap \bmL_2)^{\circ}}$ is bounded in $\bmV_{(\bmL_1\cap \bmL_2)^{\circ}}$. But we already knew  $\lim g_n v_{(\bmL_1\cap \bmL_2)^{\circ}}$ exists in $\bbP(\bmV_{(\bmL_1\cap \bmL_2)^{\circ}}\oplus \Q)$, hence the limit has to be in $\bmV_{(\bmL_1\cap \bmL_2)^{\circ}}$.
\end{proof}

\subsection{Relation between these two compactifications}

Given a sequence $(g_n)$ in $\rmG$ such that $\iota_1([g_n])$ converges, by Lemma \ref{lemMinGpStableUnder_g_n} we find the minimal element $\bmL_0$ in $\scrA(g_n)$  and $b\in \rmG$ such that $\lim g_n v_{\bmL_0} = bv_{\bmL_0}$.

\begin{prop}\label{propAlgeConvImplyMeasConv}
Under the assumption above, $\lim (g_n)_* \mu_{\rmH} = b\mu_{\rmL_0}$.
\end{prop}

\begin{proof}
It suffices to show that for any subsequence $n_k$, there exists a further subsequence $n_{k_i}$ such that 
$\lim_i g_{n_{k_i}} \mu_{\rmH} = b\mu_{\rmL_0}$. Therefore we shall feel free to pass to a subsequence whenever necessary.

So, as in Theorem \ref{thmEMSequi1}, we may assume that $g_n=\delta_n \gamma_n h_n$ with $\delta_n \to \delta_{\infty}$, that there exists a connected $\Q$-subgroup $\bmL\leq \bmG$ containing $\gamma_n \bmH \gamma_n^{-1} $ for all $n$ and that no proper subgroup of $\bmL$ contains  $\gamma_n \bmH \gamma_n^{-1} $ for infinitely many $n$. 
Hence by Corollary \ref{coroTranFullComponents}, $g_n\mu_{\rmH}$ converges to $\delta_{\infty}\mu_{\rmL}$.

By Theorem \ref{thmEMSfocusing} and passing to a subsequence, there exists 
    $d\in \bmX(\bmH,\bmL)\cap \Gamma $,
    $z_n \in (\bmZ_{\bmG}\bmH)^{\circ}\cap \Gamma$, and $l_n \in \bmL\cap \Gamma $ such that $\gamma_n= l_n d z_n$.
    
    Let $\bmL':= d^{-1}\bmL d =  (l_n^{-1}\gamma_nz_n^{-1})^{-1}\bmL (l_n^{-1}\gamma_nz_n^{-1})$, which contains $\bmH$.
    
    Then
    \begin{equation*}
        g_nv_{\bmL'}= \delta_n l_n d z_n h_n v_{\bmL'} = (\delta_n d) (d^{-1}l_nd) z_n v_{\bmL'}.
    \end{equation*}
    By our assumption $(\bmZ_{\bmG}\bmH)^{\circ} \subset \bmH$, so 
    \begin{equation*}
         g_nv_{\bmL'}=\delta_n d v_{\bmL'} \text{ is bounded }
         \implies 
         \bmL' \in \scrA_0(g_n) \implies \bmL_0\subset \bmL'.
    \end{equation*}
    Also we know that $g_n\mu_{\rmH}$ converges to $\delta_{\infty}\mu_{\rmL}= \delta_{\infty}d \mu_{\rmL'}$.
    
    On the other hand, as $\lim g_n v_{\bmL_0}= bv_{\bmL_0}$, we can find $l^0_n\in \bmL_0$ and $b_n \in \rmG$ such that $g_n=b_nl^0_n$ and $b_n\to b$.
    Hence $g_n\rmH\Gamma \subset b_nl^0_n \rmL_0\Gamma=b_n\rmL_0\Gamma$. So $\delta_{\infty}d\rmL'\Gamma/\Gamma \subset b\rmL_0\Gamma/\Gamma$.
    \begin{equation*}
        \begin{aligned}
               & \delta_{\infty} d\rmL' \subset b\rmL_0\Gamma 
                \implies\, \exists\, \gamma_0,\,  \delta_{\infty}d\rmL'\subset b \rmL_0\gamma_0
                \implies\, \exists \,l_0,\, \delta_{\infty}d=bl_0\gamma_0 \\
                \implies & bl_0\gamma_0\rmL'\subset b\rmL_0\gamma_0
                \iff \gamma_0\rmL'\gamma_0^{-1}\subset \rmL_0
        \end{aligned}
    \end{equation*}
     But we already knew $\bmL_0\subset \bmL'$, so $\gamma_0 \rmL'\gamma_0=\rmL_0=\rmL'$ and
     \begin{equation*}
         \delta_{\infty}d \mu_{\rmL'}= bl_0\gamma_0 \mu_{\rmL'} =b\mu_{\rmL_0}.
     \end{equation*}
     And the proof is complete.
\end{proof}

Now we are ready to define a map  $\pi_0$ from $X_1$ to $X_0$. For any $x\in X_1$, choose a sequence $(g_n)$ such that $\lim g_n\oplus v_{\bmL} =x$, then $\pi_0(x):=\lim g_n \mu_{\rmH}$. By Proposition \ref{propAlgeConvImplyMeasConv} above, this is well-defined and does not depend on the choice of $(g_n)$.

\begin{lem}\label{lemAlgeConvImplyMearConv2}
$\pi_0$ is continuous.
\end{lem}

\begin{proof}
It suffices to show that if $x_n\to x_{\infty}$ in $X_1$, then $\pi_0(x_n)\to \pi_0(x_{\infty})$ in $X_0$.
For each $x_n$ we choose a sequence $(g^n_i)_i$ such that 
$\lim_i g^n_i \oplus v_{\bmL} =x_n$. Then we can find $i_n$ such that for all $i_n'\geq i_n$, $x_{\infty}=\lim g^n_{i'_n}\oplus v_{\bmL}$.

On the other hand, by passing to a subsequence we may assume that $\lim \pi_0(x_n)$ exists. By definition, $\pi_0(x_n) = \lim_i g^n_i \mu_{\rmH}$. Note that the space of probability measures equipped with the weak-$*$ topology has a countable basis. Hence we can find $j_n$ such that for all $j_n'\geq j_n$, $\lim \pi_0(x_n) =\lim_n g^n_{j'_n}\mu_{\rmH}$. Taking $k_n:=\max\{i_n,j_n\}$ finishes the proof.

\end{proof}

\section{Equidistribution on varieties}\label{secBigEquiOnVarie}

\subsection{Equidistribution on smooth varieties with nice boundaries}\label{secEquiInGeometry}

The set-up is the following (see \cite[Section 4.2]{ChamTschin10}). Note that Cartier divisors and Weil divisors are equivalent in our setting, so we will just call them divisors for simplicity.
\begin{itemize}\label{setupAlgebraic}
    \item $\bmX$ is a smooth projective variety of equal dimension defined over $\R$;
    \item $\rmD$ is an effective divisor on $\bmX$ whose base change to $\C$ has strict normal crossing, write $\rmD=\sum \rmD_{\alpha}$ as sum of irreducible divisors (over $\C$);
    \item $\rmf_{\rmD}$ is the canonical section of the line bundle $\calO_{\bmX}\rmD$;
    \item $\rmU$ is the complement of $\supp(\rmD)$ in $\bmX$;
    \item $\omega_{\bmX}$ is the canonical line bundle on $\bmX$. Its sections consist of top-degree differential forms;
    \item For $V$ an open set in $\bmX$, each local section $\omega\in \Gamma(V,\omega_{\bmX})$ gives rise to a measure on $V(\R)$ which we denote by $|\omega|$;
    \item Let $\omega_{\bmX}\rmD$ be the line bundle whose local sections consist of those meromorphic top-degree differential forms whose poles can be cancelled by the zeros of $\rmD$, thought of as a Cartier divisor;
    \item Assume that $\omega_{\bmX}\rmD$ is equipped with a \textit{smooth metric}, by which we mean the following. Forget the algebraic structure and move to the category of smooth manifold by taking the real points. A metric is a collection of functions 
    \begin{equation*}
        \norm{\cdot}_x : \omega_{\bmX}\rmD(x) \longrightarrow \R_{\geq 0}
        \text{ for every } x\in \bmX(\R)
    \end{equation*}
    where $\omega_{\bmX}\rmD(x)$ denotes the fibre at $x$,
     such that $ \norm{l}_x=0$ iff $l=0$ for any $l\in \omega_{\bmX}\rmD(x)$; $\norm{al}_x=|a|\cdot \norm{l}_x$ for every $a\in\R$ and $l\in \omega_{\bmX}\rmD(x)$; $x\mapsto \norm{l(x)}_x$ is a smooth function for every $V\subset \bmX(\R)$ open and $l\in \Gamma(V,\omega_{\bmX}\rmD)$. We shall drop the dependence on $x$ from the notation if no confusion might arise;
     \item Depending on this metric we define a locally finite measure $\tau_{(\bmX,\rmD)}$. Let $\omega$ be a local section of $\omega_{\bmX}$ then locally $\tau_{\bmX,\rmD}$ is represented by 
     $\frac{|\omega|}{\norm{\omega\otimes \bmf_{\rmD}}}$;
     \item $\rmL$ is another effective divisor whose support contains $\supp(\rmD)$ and $\calO_{\bmX}(\rmL)$ is endowed with a smooth metric. Write $\rmL =\sum \lambda_{\alpha} \rmD_{\alpha} + \rmL'$ for some effective divisor $\rmL'$ whose support does not contain any $\supp(\rmD_{\alpha})$;
     \item Define 
     $\height: \rmU(\R) \to \R_{>0}$ by $\height(x):=\norm{\rmf_{\rmL}(x)} ^{-1}$ and
     \begin{equation*}
         B_{R}:=\left \{x\in \rmU(\R) \, \bigg\vert\, \norm{\rmf_{\rmL}(x)} \geq \frac{1}{R}\right\}=
         \left \{x\in \rmU(\R) \, \vert\, \height(x) \leq {R}\right\}
         ,
     \end{equation*}
     a compact set contained in $\rmU(\R)$.
\end{itemize}

It is proved in \cite[Corollary 4.8]{ChamTschin10} that 

\begin{thm}
The family $\{B_R\}$ is nice (see section \ref{secEquiImplyCount}) and the limit of the measures
\begin{equation*}
    \frac{1}{\tau_{\bmX,\rmD}(B_R)} \tau_{\bmX,\rmD}\big\vert_{B_R}
\end{equation*}
exists as $R$ tends to $+\infty$. Moreover the limit measure is a sum of continuous measures, each of which is supported on a manifold defined by some connected components of the $\R$-points of the intersection of certain irreducible components of $\rmD_{\C}$. 
\end{thm}

By saying that a measure on a manifold is \textit{continuous}, we mean that in local charts it is absolutely continuous with respect to the Lebesgue measure and its Radon-Nikodym derivative is a continuous function.

Their proof also yields the following, which will be actually applied. For each connected component (in the Hausdorff topology) $U$ of $\rmU(\R)$, 

\begin{thm}\label{thmEquiBallsOnVariety}
The family $\{B_R\cap U\}$ is nice and the limit of the measures
\begin{equation*}
    \frac{1}{\tau_{\bmX,\rmD}(B_{R}\cap U)} \tau_{\bmX,\rmD}\big\vert_{B_{R}\cap U}
\end{equation*}
exists as $R$ tends to $+\infty$. Moreover the limit measure is a sum of continuous measures, each of which is supported on a manifold defined by some connected components of the $\R$-points of the intersection of certain irreducible components of $\rmD_{\C}$. 
\end{thm}

Actually this limit measure has a more precise description in \cite[Corollary 4.8]{ChamTschin10} but we are content with this weaker statement.
Let us give a more precise description of the asymptotic of $\tau_{\bmX,\rmD}(B_{R}\cap U)$ though. 

Let $\scrB$ be an index set of irreducible components of $\rmD_{\C}$ and $\scrP(\scrB)$ be the collection of subsets of $\scrB$. For $A\in \scrP(\scrB)$, let $\rmD_A:= \cap_{\alpha\in A} \rmD_{\alpha}$.
Let $\Gal$ be the Galois group of $\C$ over $\R$. It acts on $\scrP(\scrB)$ and let $\scrP(\scrB)^{\Gal}$ be the subsets of $\scrB$ that are $\Gal$-stable. 

Define the analytic Clemens complex

\begin{equation*}
    \begin{aligned}
             \scrC^{an}_{\R}(\rmD):=\left\{
               (A,Z) \:\middle\lvert\:
             \parbox{80mm}
              {
              \raggedright
              $A \in \scrP(\scrB)^{\Gal},\,Z$  is a $\C$-irreducible component of $\rmD_A$ defined over $\R$, $Z(\R)$ is non-empty
              }
             \right\}
    \end{aligned}
\end{equation*}
equipped with the partial order $\leq$ defined by
\begin{equation*}
    (A,Z)\leq (A',Z') \iff A\subset A',\,Z\supset Z'.
\end{equation*}

For $(A,Z)\in \scrC^{an}_{\R}(\rmD)$, define
\begin{equation*}
    \dim(A,Z):=
    \max\left\{
    n\,\middle\lvert\,
    \exists (A_0,Z_0)\lneq...\lneq(A_n,Z_n)=(A,Z)
    \right\}.
\end{equation*}
The dimension of a poset (poset = partially ordered set) would then be the maximum of dimensions of elements contained in it.

Now we take $U$ to be a union of some connected components (in the Hausdorff topology) of $\rmU(\R)$. When $U=\rmU(\R)$, definitions below would coincide with the corresponding one in \cite{ChamTschin10}.

Define $\scrC^{an}_{\R,U}(\rmD)$ to be a sub-poset of $\scrC^{an}_{\R}(\rmD)$ by
\begin{equation*}
    \begin{aligned}
             \scrC^{an}_{\R,U}(\rmD):=\left\{
               (A,Z)\in \scrC^{an}_{\R}(\rmD) \:\middle\lvert\:
              Z(\R)\cap \overline{U}\neq \emptyset
             \right\}.
    \end{aligned}
\end{equation*}

Define 
\begin{equation*}
    \sigma_U :=
    \left\{
    \frac{d_{\alpha}-1}{\lambda_{\alpha}} \,\middle\lvert\,
    \alpha\in \scrB, \rmD_{\alpha}(\R)\cap \overline{U} \neq \emptyset
    \right\}
\end{equation*}
where $\rmD_{\alpha}(\R)$, by definition, is $\rmD_{\alpha}(\C)\cap \bmX(\R)$ in the case when $\rmD_{\alpha}$ may not be defined over $\R$. We collect the indices where the maximum is attained to form $\scrB_{max,U}\subset\scrB$.

Define
\begin{equation*}
    \scrC^{an}_{\R,(L,D),U}(\rmD) :=
    \left\{
    (A,Z)\in  \scrC^{an}_{\R,U}(\rmD) \,\middle\lvert\,
    A\subset \scrB_{max,U}
    \right\}
\end{equation*} and $b_U:=\dim \scrC^{an}_{\R,(L,D),U}$. 

The following is a (again we are content with a weaker statement) variant of \cite[Theorem 4.7]{ChamTschin10} which follows from the same proof. For the sake of notation let $b'_U:= b_U$ if $\sigma_U>0$ and $b'_U:=b_U+1$ if $\sigma_U=0$.

\begin{thm}\label{thmAsymBalls}
Assumptions same as above. Then there exists a constant $C>0$ such that
\begin{equation*}
    \lim_{R\to\infty}\frac{\tau_{\bmX,\rmD}(B_R\cap U)}{CR^{\sigma_U}\ln(R)^{b'_U-1}} =1.
\end{equation*}
\end{thm}

\subsection{Resolution of singularities}

\begin{thm}\label{thmResolution}
There exists a smooth projective $\bmG$-variety $\bmX_2$ over $\R$ with a $\bmG$-equivariant proper morphism $\pi_1:\bmX_2 \to \bmX_1$ such that :
\begin{enumerate}
    \item let $\bmU_2:=\pi_1^{-1}(\bmG/\bmH)$, then $\pi_1\vert_{\bmU_2}$ is an isomorphism;
    \item let $\rmD_2:= \pi_1^{-1}(\rmD_1)$, then $(\rmD_2)_{\C}$ is a strict normal crossing divisor. 
\end{enumerate}
\end{thm}

See \cite{BierMil97} and \cite{EncVill98} for instance. 
And \cite[Proposition 3.9.1]{kollar07} explains how to lift the algebraic action of an algebraic group.

\subsection{Divisor of the G-invariant top-degree differential form}\label{secHaarMeasIsGeometric}

The result of this subsection is valid more generally for $\bmG$ semisimple and $\bmH$ with $\bmZ_{\bmG}\bmH/\bmZ(\bmH)$ being $\Q(i)$-anisotropic. $\bmZ(\bmH)$ is defined as the center of $\bmH$.

We verify that the Haar measure $\mu_{\rmG/\rmH}$ can be  constructed from $\rmD_2$ via the method of Section \ref{secEquiInGeometry}. The whole subsection is devoted to proving the following.

\begin{prop}\label{propHaarisGeometric}
There exists an effective divisor $\rmD_2'$ whose support is equal to $\supp (\rmD_2)$ and a smooth metric on $\calO_{\bmX}(\rmD_2')$ such that $\tau_{\bmX_2,\rmD'_2}\vert_{\rmG/\rmH}=\mu_{\rmG/\rmH}$.
\end{prop}

 To prove this it suffices to show that 
 \subsubsection*{Claim}
 The $\bmG$-invariant top-degree differential form $\omega_0$ on $\bmG/\bmH$, when viewed as a meromorphic section of the canonical line bundle, has poles at every $\rmD_{\alpha}$ for every irreducible component $\rmD_{\alpha}$.
 
 Once this is proved, $\rmD_2'$, defined as $-\divisor(\omega_0)$, is an effective divisor whose support is equal to $\supp(\rmD_2)$ and $\omega_0$ becomes a global nowhere vanishing section of $\omega_{\bmX_2}(\rmD'_2)$. And we simply define a metric on $\omega_{\bmX_2}(\rmD'_2)$ by imposing $\norm{\omega_0}\equiv 1$. This metric is smooth and $\tau_{\bmX_2,\rmD_2'}$ is equal to the Haar measure $|\omega_0|$ on $\rmG/\rmH$.
 
 To prove the claim it suffices to show that 
 for each irreducible component $\rmD_{\alpha}$ of $\rmD_{\C}$, there exists $x\in \rmD_{\alpha}(\C)$ and a section $D$ in the dual of $\omega_{\bmX}$ such that $D_{x}=0$ but $\la D, \omega_0 \ra_{x_i} $ is bounded away from $0$ from below for a sequence $(x_i)$ converging to $x$. 
 As $\omega_0$ is $\bmG$-invariant, it is allowed to replace $x$ by $gx$ for some $g\in \bmG$.
 Let us note this approach has been applied to the setting of bi-equivariant compactification of $\bmG$ in \cite[Section 2]{ChamTschin02} and the proof below is inspired by theirs.

As $\bmU_2(\C)$ is dense in $\bmX_2(\C)$ in the Hausdorff topology, we may find a sequence $(g_n)$ from $\bmG(\C)$ such that $\lim g_n\oplus v_{\bmL} =x$. 

Now apply Theorem \ref{thmEMSnondiverg1} to
\begin{itemize}
    \item $\bmG':= \rmR_{\Q(i)/\Q}\bmG$;
    \item $\bmH':= \rmR_{\Q(i)/\Q}\bmH$;
    \item $\Gamma'$ to be a lattice commensurable with $\bmG'(\Z)$;
    \item $(g_n)\subset \bmG(\C)\cong \bmG'(\R)$.
\end{itemize}
Note that under our condition \ref{condition 1}, we still have $(\bmZ_{\bmG'}\bmH')^{\circ}\subset \bmH'$.
We find a  bounded sequence $(\delta_n)$ in $\bmG(\C)$, a sequence $(\gamma_n)$ in $\Gamma'$, which is commensurable with $\bmG(\Z[i])$ when identifying $\bmG'(\R)\cong \bmG(\C)$, and $(h_n)$ in $\bmH(\C)$ such that $g_n=\delta_n \gamma_n h_n$. 
As $g_n\oplus v_{\bmL} = \delta_n \gamma_n v_{\bmL}$, we may and do assume that $h_n=id$. By passing to a subsequence also assume that $\delta_n$ converges to $\delta_{\infty}$. Now $x= \lim \delta_{\infty} \gamma_n\oplus v_{\bmL}$, so we just assume that $\delta_n=\delta_{\infty}$. 
Replacing $x$ by $\delta_{\infty}^{-1}x$, we further assume $g_n=\gamma_n$. 

Let $\frakg$ (resp. $\frakh$) denote the Lie algebra of $\bmG$ (resp. $\bmH$). Let $m:=\dim \bmG/\bmH$. 
By passing to a further subsequence we find an $m$-dimensional $\Q(i)$-subspace $W$ in $\frakg_{\Q(i)}$ such that $W \cap \Ad \gamma_n \frakh =\{0\}$ for all $n$. 

Our strategy of proving the claim is to construct a section $D$ of $\bigwedge^{\text{top}}\scrT_{\bmX_2}$, where $\scrT_{\bmX_2}$ denotes the tangent bundle on $\bmX_2$, such that $\la D, \omega_0\ra_{x} \neq 0$ but $D$ vanishes at $x$. 
It suffices to show that the quantity $|\la D, \omega_0\ra_{\gamma_n\oplus v_{\bmL}} | > \delta_0$ for some $\delta_0>0$ independent of $n$.

Indeed for a vector $v \in \frakg$, define $D_v$ by
\begin{equation}
    \la D_v,f\ra_x  := \frac{\diff}{\difft}\bigg\vert_{t=0} f(\exp(tv)x)
\end{equation}
for every function $f$ defined locally at $x$.
For $\bmv=v_1\wedge v_2 \wedge ... \wedge v_m$, we let $D_{\bmv}:= D_{v_1}\wedge...\wedge D_{v_m}$. Then we choose $D$ to be $D_{\bmw}$ for some set $\{w_1,...,w_m\}$ of $\Q(i)$-basis of $W$.

Let $\pi_{\frakh}$ denote the projection $\frakg\to \frakg/\frakh$. We fix a set of basis $\{x_1,...,x_m\}$ of $(\frakg/\frakh)^*$. 
Choose a $\Q(i)$-subspace $\frakh'$ of $\frakg$ complementary to $\frakh$ and we pull back $x_i$ to be linear functionals $x_i'$ on $\frakh'$. Identify a small neighborhood of $0$ in $\frakh'(\C)$ with a small neighborhood of $[e]$ in $\bmG/\bmH$ via $v \mapsto [\exp(v)]$. 
Under this identification, each $x_i'$ becomes a function $\widetilde{x}_i$ on a small neighborhood of $[e]$ in $\bmG/\bmH$. 
Then $\omega_0$ is a $\bmG$-invariant differential form whose localization at identity coset is equal to some non-zero multiple of $\diff{\wtx}_1\wedge...\wedge \diff{\wtx}_m$. Without loss of generality, we just assume that they are equal at the identity coset. In light of the following lemma, this definition does not depend on the choice of $\frakh'$.

\begin{lem}
$\la D_v , \diff{\wtx}_i \ra_{[e]} = \la\pi_{\frakh}(v) ,x_i \ra$.
\end{lem}

\begin{proof}
Write $v= v_{\frakh'}+v_{\frakh} = v_1 + v_2$, then $D_{v}=D_{v_1}+D_{v_2}$ and $(D_{v_2})_{[e]}=0$.
\begin{equation*}
\begin{aligned}
       &\la D_v , \diff \wtx_i \ra_{[e]} = 
    \la D_{v_1} , \diff \wtx_i \ra_{[e]}  
     = \frac{\diff}{\difft}\bigg\vert_{t=0} \wtx_i([\exp{(v_1t)}])\\
     =& \frac{\diff}{\difft}\bigg\vert_{t=0}
     x'_i(v_1t)
    =\la \pi_{\frakh}(v_1) , x_i \ra = \la \pi_{\frakh}(v) ,x_i \ra.
\end{aligned}
\end{equation*}
\end{proof}

\begin{lem}
$\la D_{\bmw} , \omega_0 \ra_{[g]} = 
\det \left( 
\la  \pi_{\frakh}( \Ad g^{-1}w_i ), x_j\ra
\right)_{i,j}$.
\end{lem}

\begin{proof}
Note that $(\omega_0)_{[g]} = \text{L}_{g^{-1}}^*(\omega_0)_{[e]}$ where $\text{L}_g$ denotes the left action of $\rmG$ on $\rmG/\rmH$.
So we just need to show that 
$\la D_{w_i} ,  \text{L}_{g^{-1}}^*\diff\wtx_j \ra_{[g]} = \la  \pi_{\frakh} (\Ad g^{-1}w_i) , x_j\ra$.
Indeed
\begin{equation*}
\begin{aligned}
           \LHS &= \frac{\diff}{\difft}\bigg\vert_{t=0}
    \text{L}_{g^{-1}}(\wtx_j)(\exp(w_it)[g]) 
    =\frac{\diff}{\difft}\bigg\vert_{t=0}
    (\wtx_j)(g^{-1}\exp(w_it)g[e]) \\
    &=\la 
    D_{(\Ad g^{-1})w_i} , \diff \wtx_j
    \ra_{[e]},
\end{aligned}
\end{equation*}
which is equal to the right hand side by the lemma above.
\end{proof}

Now we can prove our claim. As $w_i$'s are $\Q(i)$-vectors and $\gamma_n$'s  are contained in $\Gamma'$, which is commensurable with $\bmG(\Z[i])$, we have that 
$\left\{\Ad \gamma_n^{-1} \cdot w_i \right\}_{n,i}$ are vectors with bounded denominators. As $x_j$'s, $\pi_{\frakh}$ are all defined over $\Q(i)$, we have that the entries of the matrices $ \left( 
\la \pi_{\frakh} \Ad \gamma_n^{-1}w_i ,
    x_j \ra
\right)_{i,j}$ as $n$ varies are $\Q(i)$-vectors with bounded denominators. Therefore there exists a constant $\ep_0 >0$ such that 
\begin{equation*}
\la D_{\bmw},\omega_0\ra_{[\gamma_n]}=
    \left|
\det \left(
\la \pi_{\frakh} \Ad \gamma_n^{-1}w_i ,
    x_j
\ra \right)_{i,j}
\right|
\end{equation*}
 is either $0$ or at least $\ep_0$. By our choice of $W$, it is not equal to $0$, so we are done.

\subsection{One source of examples of heights}\label{secExamHt}

By partition of unity one can construct smooth metrics for any line bundle. On the other hand, it is not clear that every proper (in the sense of topology) real algebraic function $l: \rmG/\rmH \to \R_{>0}$ can be obtained this way. In this subsection we explain that those arising from a representation and an Euclidean norm (a variant: replacing all the number $2$ below by any other fixed positive number) does arise this way.

Let $\rho_0:\bmG \to \GL(V_0)$ be a representation of $\bmG$ over $\Q$. Let $v_0\in V_0(\Q)$ be a vector whose stabilizer in $\bmG$ is exactly equal to $\bmH$. 
Fix an isomorphism $V_0(\R)\cong \R^N$ for some $N$ and write 
\begin{equation*}
   \norm{x}:= \left(
    x_1^2+...+x_N^{2}
    \right)^{1/2}.
\end{equation*}
Consider the function $l:\bmG/\bmH(\R)\to \R$ defined by
\begin{equation*}
    l([g]):= \norm{g\cdot v_0}.
\end{equation*}
Let $\bmX_{v_0}$ be the compactification of $\bmG\cdot v_0$ in $\bbP(V_0\oplus \Q)$.
By considering the diagonal embedding $\bmG/\bmH \to \bmX_1 \times \bmX_{v_0}$ and replacing $\bmX_1$ by the compactification of this one, we may assume that there exists $\pi_{v_0}: \bmX_1 \to \bmX_{v_0}$ that is $\bmG$-equivariant.

\begin{lem}\label{lemNormRepIsGeometric}
Under the assumption made in the last paragraph, there exists a line bundle $L$ on $\bmX_2$ whose support contains $\rmD_2$ and a smooth metric such that $l$ coincides with the corresponding $\height$ defined as in Section \ref{setupAlgebraic}.
\end{lem}

\begin{proof}
Indeed the function $l$ may be viewed as obtained from certain metric on the pull-back of the line bundle $\calO_{\bbP(\bmV_0\oplus \Q)}(1)$ whose evaluation at the canonical section associated with the divisor $\bbP(\bmV_0\oplus \{0\})$ is non-vanishing. As there is a morphism from $\bmX_2$ to $\bmX_{v_0}$ we can further pull back the metrized line bundle and the effective divisor to $\bmX_2$. It only remains to check that the metric is indeed smooth. 
Though $\bmX_{v_0}$ may not be smooth by itself, the metric is smooth in the sense that it can be extended to a smooth metric in the ambient space. 
Therefore its pull-back to $\bmX_2$ also satisfies the same property. But $\bmX_2$ is smooth by itself, hence this metric is also smooth in the usual sense. And the proof completes.
\end{proof}

\section{Conclusion of the proof}\label{secWrapUp}

Now we prove Proposition \ref{propEquionProb}, Corollary \ref{CoroCounting2}, and Theorem \ref{thmEquiorCountwithWeights} in this section.

Let $X_1:=\pi_1^{-1}(X_0)$, $X_2:=\pi_2^{-1}(X_1)$ and $X_{v_0}:= \overline{\rmG\cdot v_0}$. Thus we have the following $\rmG$-equivariant picture:
\[
   \begin{tikzcd}
        && X_2 \arrow[d, "\pi_2"]& \\
        && X_1 \arrow[dl, "\pi_1"'] \arrow[dr,"\pi_{v_0}"]& \\
        \Prob(\rmG/\Gamma) & X_0\arrow[l, "\text{inclusion}"]&& X_{v_0}.
   \end{tikzcd}
\]
Let $U_2:=(\pi_1\circ\pi_2)^{-1}U_0$ and $$\lambda_2:=\lim_{R\to\infty} \frac{1}{\mu_{\rmG/\rmH}(B_R\cap U_2)}\cdot \mu_{\rmG/\rmH}\vert_{B_R\cap U_2}$$ in $\Prob(X_2)$, which exists by Theorem \ref{thmEquiBallsOnVariety} with the help of Theorem \ref{thmResolution}, Proposition \ref{propHaarisGeometric} and Lemma \ref{lemNormRepIsGeometric}. Write $\lambda_0:=(\pi_1\circ\pi_2)_*\lambda_2$. Recall 
$\Ave: \Prob(\Prob(\rmG/\Gamma)) \to \Prob(\rmG/\Gamma)$ from Section \ref{secAReformulation}. Then we have
\begin{equation*}
    \lim_{R\to \infty} 
    \frac{1}{\mu_{\rmG/\rmH}(B_R)} \int_{B_R} g_* \mu_{\rmH} \,\mu_{\rmG/\rmH}([g]) = \Ave(\lambda_0) = \int \mu \, \lambda_0(\mu)
\end{equation*}
also exists.

This proves Proposition \ref{propEquionProb} and hence Theorem \ref{thmCounting1}. Together with Theorem \ref{thmAsymBalls}, Corollary \ref{CoroCounting2} follows. Now we explain how to prove Theorem \ref{thmEquiorCountwithWeights} in a similar way.

Now let $\phi$ be a continuous function on $X_2$, we need to know where does the integration of $\phi$ against the measure indicated in Theorem \ref{thmEquiorCountwithWeights} converges to. Pushing downwards our discussion below to $X_{v_0}$ gives the proof.

Using similar proof as in Proposition \ref{propEquiImplyCount}, we can show the following (proof omitted). Note that the assumption made in Proposition \ref{propEquiImplyCount} has been verified in Section \ref{secBigEquiOnVarie}.

\begin{prop}
The following limit 
\begin{equation*}
    \lim_{R \to + \infty} 
    \frac{1}{\mu_{\rmG/\rmH}(B_R)} 
    \int_{B_R} \phi ([g]) g_* \mu_{\rmH} \, \mu_{\rmG/\rmH}([g])
\end{equation*}
exists and  there exists a positive continuous function $f^{\phi}_{\infty}$ on $\rmG/\Gamma$ such that 
\begin{equation*}
    f^{\phi}_{\infty}([g]) 
    = \lim_{R\to +\infty} 
    \frac{
    \sum_{\gamma \in \Gamma / \Gamma \cap \rmH} (1_{B_R}\cdot \phi) (g\gamma \cdot v_0)
    }
    {\mu_{\rmG/\rmH}(B_R)}.
\end{equation*}
\end{prop}

Hence each $[g]\in \rmG/\Gamma$,
\begin{equation*}
    \phi \mapsto 
     \frac{f^{\phi}_{\infty}([g]) }{f^{1}_{\infty}([g]) }
\end{equation*}
defines a positive bounded linear functional on $C(X_2)$ sending $1$ to $1$. By Riesz representation theorem, there exists $\nu_{[g]}\in \Prob(X_2)$ such that 
\begin{equation*}
    \nu_{[g]} = \lim_{R\to +\infty} 
    \frac{1}{f^{1}_{\infty}([g])\mu_{\rmG/\rmH}(B_R)} 
    \sum_{\gamma \in \Gamma /\Gamma \cap \rmH} \delta_{g\gamma \cdot v_0}.
\end{equation*}
Pushing this measure forward to $X_{v_0}$ yields Theorem \ref{thmEquiorCountwithWeights}.

\bibliographystyle{amsalpha}
\bibliography{ref}

\end{document}